\documentclass[11pt,a4paper]{amsart}

\usepackage{amsmath,amssymb, amsbsy}
\usepackage{color,psfrag}
\usepackage[dvips]{graphicx}
\usepackage{enumerate}

\def\neweq#1{\begin{equation}\label{#1}}
\def\endeq{\end{equation}}
\def\eq#1{(\ref{#1})}

\newtheorem{theorem}{Theorem}

\newtheorem{lemma}[theorem]{Lemma}

\begin{document}

\title[]{A measure of the torsional performances\\ of partially hinged rectangular plates}



\author[Elvise BERCHIO]{Elvise BERCHIO}
\address{\hbox{\parbox{5.7in}{\medskip\noindent{Dipartimento di Scienze Matematiche, \\
Politecnico di Torino,\\
        Corso Duca degli Abruzzi 24, 10129 Torino, Italy. \\[3pt]
        \em{E-mail address: }{\tt elvise.berchio@polito.it}}}}}
\author[Davide BUOSO]{Davide BUOSO}
\address{\hbox{\parbox{5.7in}{\medskip\noindent{Dipartimento di Scienze Matematiche, \\
Politecnico di Torino,\\
        Corso Duca degli Abruzzi 24, 10129 Torino, Italy. \\[3pt]
        \em{E-mail address: }{\tt davide.buoso@polito.it}}}}}
\author[Filippo GAZZOLA]{Filippo GAZZOLA}
\address{\hbox{\parbox{5.7in}{\medskip\noindent{Dipartimento di Matematica,\\
Politecnico di Milano,\\
   Piazza Leonardo da Vinci 32, 20133 Milano, Italy. \\[3pt]
        \em{E-mail addresses: }{\tt
          filippo.gazzola@polimi.it}}}}}





\renewcommand{\ge }{\geqslant}
\renewcommand{\geq }{\geqslant}
\renewcommand{\le }{\leqslant}
\renewcommand{\leq }{\leqslant}

\newcommand{\R}{{\mathbb R}}
\newcommand{\N}{{\mathbb N}}
\newcommand{\Z}{{\mathbb Z}}
\newcommand{\C}{{\mathbb C}}
\renewcommand{\H}{{\mathcal H}}
\newcommand{\SN}{{\mathbb S}^{N-1}}
\newcommand{\weakly}{\rightharpoonup}
\newcommand{\e }{\varepsilon}
\newcommand{\Di}{{\mathcal D}^{1,2}(\R^N)}
\newcommand{\alchi}{\raisebox{1.7pt}{$\chi$}}
\newcommand{\dive }{\mathop{\rm div}}
\newcommand{\sgn}{\mathop{\rm sgn}}

\newcommand{\HH}{\mathcal{H}}
\newcommand{\F}{\mathcal{F}}
\newcommand{\GG}{\mathcal{G}}

\newcommand{\eps}{\varepsilon}
\newcommand{\smx}{\sin(mx)}

\newcommand{\serie}{\sum_{m=1}^{+\infty}}
\newcommand{\integ}{\int_\Omega}
\newcommand{\ipi}{\int_0^\pi}
\newcommand{\il}{\int_{-\ell}^\ell}

\begin{abstract}
We introduce a new functional, named gap function, which measures the torsional instability of a partially hinged rectangular plate, aiming to model the deck of a bridge. Then we test the performances of the gap function on a plate subject to two kinds of loads: one modeled by a force concentrating on one of the free edges and one in resonance with the eigenmodes.\par
\end{abstract}

\maketitle

\allowdisplaybreaks


\section{Introduction}
From the Federal Report \cite{ammann} on the Tacoma Narrows Bridge collapse (see also \cite{scott}), we learn that the most dangerous oscillations
for the deck of a bridge are the torsional ones, which appear when the deck rotates around its main axis; we refer to \cite[\S 1.3,1.4]{book} for a historical
survey of several further collapses due to torsional oscillations. These events naturally raise the following question:

\centerline{\bf is it possible to measure the torsional performances of bridges?}

Following \cite{fergaz} we model the deck of a bridge as a long narrow rectangular thin plate $\Omega$ hinged at two opposite edges and free on
the remaining two edges; this well describes the deck of a suspension bridge which, at the short edges, is supported by the ground.
Then we introduce a new functional, named \emph{gap function}, able to measure the torsional performances of the bridge. Roughly speaking, this functional
measures the gap between the displacements of the two free edges of the deck, thereby giving a measure of the risk for the bridge to collapse. In the present paper
we explicitly compute the gap function for some prototypes of external forces that appear to be the most prone to generate torsional instability in the structure.
Our theoretical and numerical results confirm that the gap function is a reliable measure for the torsional performances of rectangular plates.

\section{Variational setting and gap function definition}\label{2}

Up to scaling, we may assume that $\Omega=(0,\pi)\times(-\ell,\ell)\subset\R^2$ with $2\ell\ll\pi$. According to the Kirchhoff-Love
theory \cite{Kirchhoff,Love}, the energy $\mathbb E$ of the vertical
deformation $u$ of the plate $\Omega$ may be computed by
{\small $$
\mathbb E(u)=\int_\Omega \left(\frac 12 (\Delta u)^2+(1-\sigma)(u_{xy}^2-u_{xx}u_{yy})-fu\right)\, dx\,dy \, ,
$$}
where $0<\sigma<1$ is the Poisson ratio. The functional $\mathbb E$ is minimized on the space $H^2_*:=\Big\{w\in H^2(\Omega);\, w=0\mbox{ on }\{0,\pi\}\times (-\ell,\ell)\Big\};$
since $\Omega\subset \R^2$, $H^2(\Omega)\subset C^0(\overline{\Omega})$ so that the condition on $\{0,\pi\}\times(-\ell,\ell)$ is satisfied pointwise.
By \cite[Lemma 4.1]{fergaz}, $H^2_*$ is a Hilbert space when endowed with the scalar product

$$(u,v)_{H^2_*}:=\int_\Omega \left[\Delta u\Delta v+(1-\sigma)(2u_{xy}v_{xy}-u_{xx}v_{yy}-u_{yy}v_{xx})\right]\, dxdy\,$$
\noindent
and associated norm $\|u\|_{H^2_*}^2=(u,u)_{H^2_*}$, which is equivalent to the usual norm in $H^2(\Omega)$.
We also define $H^{-2}_*$ the dual space of $H^2_*$
and we denote by $\langle\cdot,\cdot\rangle$ the corresponding duality.
If $f\in H^{-2}_*$ we replace $\int_\Omega fu$ with $\langle f,u\rangle$. In such case, there exists a unique $u\in H^2_*$ such that\vskip-3mm
\begin{equation} \label{lax-milg}
(u,v)_{H^2_*}=\langle f,v\rangle\qquad \forall v\in H^2_*\,,
\end{equation}
and $u$ is the minimum point of the functional $\mathbb E$.
If $f\in L^2(\Omega)$ then $u\in H^4(\Omega)$ and $u$ is a strong solution of the problem\vskip-4mm
\begin{equation} \label{eq:Poisson}
\begin{cases}
\Delta^2 u=f & \text{ in } \Omega \\
u(0,y)=u_{xx}(0,y)=u(\pi,y)=u_{xx}(\pi,y)=0 &\text{ for } y\in (-\ell,\ell) \\
u_{yy}(x,\pm\ell)+\sigma u_{xx}(x,\pm\ell)=u_{yyy}(x,\pm\ell)+(2-\sigma)u_{xxy}(x,\pm\ell)=0 & \text{ for } x\in (0,\pi)\ .\\
\end{cases}
\end{equation}
See \cite{fergaz,mansfield} for the derivation of \eq{eq:Poisson}. Differently from what happens in higher space dimension or for lower order problems,
here $u\in H^2_*$ is {\em continuous} if $f$ merely belongs to $L^1$ or $H^{-2}_*$ and we can define the \emph{gap function}:
\neweq{gap}
\GG (x):=u(x,\ell)-u(x,-\ell)\qquad x\in(0,\pi)\,.
\endeq
$\GG$ measures the difference of the vertical displacements on the free edges and is therefore a measure of the
torsional response. The maximal gap is given by
\neweq{gapinfty}
\GG^\infty:=\max_{x\in[0,\pi]}\ \big|u(x,\ell)-u(x,-\ell)\big|\,.
\endeq
This gives a measure of the risk for the bridge to collapse. Clearly, $\GG^\infty$ depends on $f$ and our purpose is to compute the torsional performances of $\Omega$ for different $f$.\par
When $f=f(x)$ the explicit solution of \eq{eq:Poisson} was obtained in \cite{fergaz} by adapting the separation of variables approach
of \cite[Section 2.2]{mansfield}. In fact, a similar procedure can be used also for some forcing terms depending on $y$ such as $e^{\alpha y}g(x)$ or $yg(x)$, see
\cite[p.41]{mansfield}. Here we focus our attention on loads of the form
{\small \neweq{f2}
f=f_\alpha(x,y)=K_\alpha e^{\alpha y}g(x)\,,\quad g\in L^2(0,\pi)\quad K_\alpha:=\frac{\alpha}{2\sinh(\alpha \ell)\, \int_0^{\pi}|g(x)|\,dx}\,.
\endeq}\!
Hence, $\|f_\alpha\|_{L^1}=1$. Furthermore, we write

\neweq{serieg}
{
g(x)=\serie\gamma_m\smx\ ,\qquad\gamma_m=\frac{2}{\pi}\ipi g(x)\smx\, dx\,.
}
\endeq
We will show in Lemma \ref{limitfalpha} that $f_{\alpha}$, with $\alpha$ big, is a good approximation of a load concentrated on a long edge of the plate. This is a good model for the restoring force due to the hangers in a suspension bridge because they act close to the free edges. We prove

\begin{theorem}\label{exponential}
Assume that $f$ satisfies \eqref{f2}-\eqref{serieg} for some $\alpha\ge0$, $\alpha\not\in\mathbb{N}$. Then the unique solution
of \eqref{eq:Poisson} is given by
{\small $$u(x,y)=K_\alpha \serie\left[\frac{\gamma_m\, e^{\alpha y}}{(m^2-\alpha^2)^2}+A\cosh(my)+B\sinh(my)+Cy\cosh(my)+Dy\sinh(my)\right]\smx$$}\!
where the constants $A=A(m,\ell,\alpha)$, $B=B(m,\ell,\alpha)$, $C=C(m,\ell,\alpha)$, $D=D(m,\ell,\alpha)$ are the solutions of the systems
\eqref{ADBCsystem} provided in Section \ref{proof Th3}.
\end{theorem}
By Theorem \ref{exponential}, the gap function \eq{gap} can be computed for all $f$ satisfying \eqref{f2}-\eqref{serieg}. Here, for the sake of simplicity,
we focus on the case where the concentration occurs at the midpoint of the upper edge of the plate. Namely,
{\small \neweq{falpha}
f=f_\alpha(x,y):=\frac{\alpha\, e^{\alpha y}\sin  x}{4\, \sinh(\alpha\ell)}\, >0\mbox{ in }\Omega
\endeq}\!
so that $\|f_\alpha\|_{L^1}=1$. The unique solution $u=u_\alpha$ of \eq{eq:Poisson} with $f=f_\alpha$ as in \eq{falpha} is given by Theorem \ref{exponential} with $\gamma_1=1$ and $\gamma_m=0$ if $m\neq 1$. We set
{\small \neweq{Elalpha}
E(\ell,\alpha):=\frac{\alpha}{2(1-\alpha^2)^2}+ \frac{\alpha\, B\,\sinh \ell}{2\, \sinh(\alpha\ell)} + \frac{\alpha\, C\,\ell\,\cosh \ell}{2\, \sinh(\alpha\ell)} \,,
\endeq}
with $B=B(1,\ell,\alpha)$ and $C=C(1,\ell,\alpha)$ as in Theorem \ref{exponential}, see \eq{horrible}.
Finally, we set
{\small \neweq{El}
E(\ell):=  \frac{  \sinh^2(\ell)}{ (1-\sigma)\left[(1-\sigma)\ell +(3+\sigma)\sinh(\ell)\cosh(\ell)\right]}
\endeq}
and we prove

\begin{theorem}\label{gap_limit}
Let $u_\alpha$ be the unique solution of \eqref{eq:Poisson} with $f=f_\alpha$ as in \eqref{falpha}. Let $\GG_\alpha$ and $\GG^{\infty}_{\alpha}$ be as in \eqref{gap} and in \eqref{gapinfty} with $u=u_\alpha$. Assume \eqref{Elalpha} and \eqref{El}. Then,
{\small \neweq{Gla} \GG_\alpha\left(x\right)=E(\ell,\alpha) \,\sin x\rightarrow \overline \GG(x):=E(\ell) \,\sin x\qquad\mbox{uniformly on $[0,\pi]$ as }\alpha\to+\infty\, .
\endeq}
In particular,
{\small
\neweq{Ela}\GG^{\infty}_{\alpha}=E(\ell,\alpha)= E(\ell)-\frac{1}{\alpha}\,\frac{(1+\sigma)\cosh \ell \sinh \ell+(1-\sigma)\ell}{2 (1\!-\!\sigma)[(3\!+\!\sigma)\cosh(\ell)\sinh(\ell)+(1\!-\!\sigma)\ell]}+o\left(\frac{1}{\alpha}\right)\mbox{ as }\alpha\to+\infty\, .
\endeq}
\end{theorem}

Theorem \ref{gap_limit} gives the limit value of $\GG^{\infty}_{\alpha}$ as $\alpha\to\infty$ but it does not clarify the behavior of the map $\alpha\mapsto\GG^{\infty}_{\alpha}$. Figure \ref{Ela2} supports the conjecture that this map is strictly increasing. Hence, if the same total load
approaches the free edges, then the gap function increases: this validates $\GG^\infty$ as a measure of the torsional performances.
\begin{figure}[ht]
\begin{center}
\includegraphics[width=0.49\textwidth]{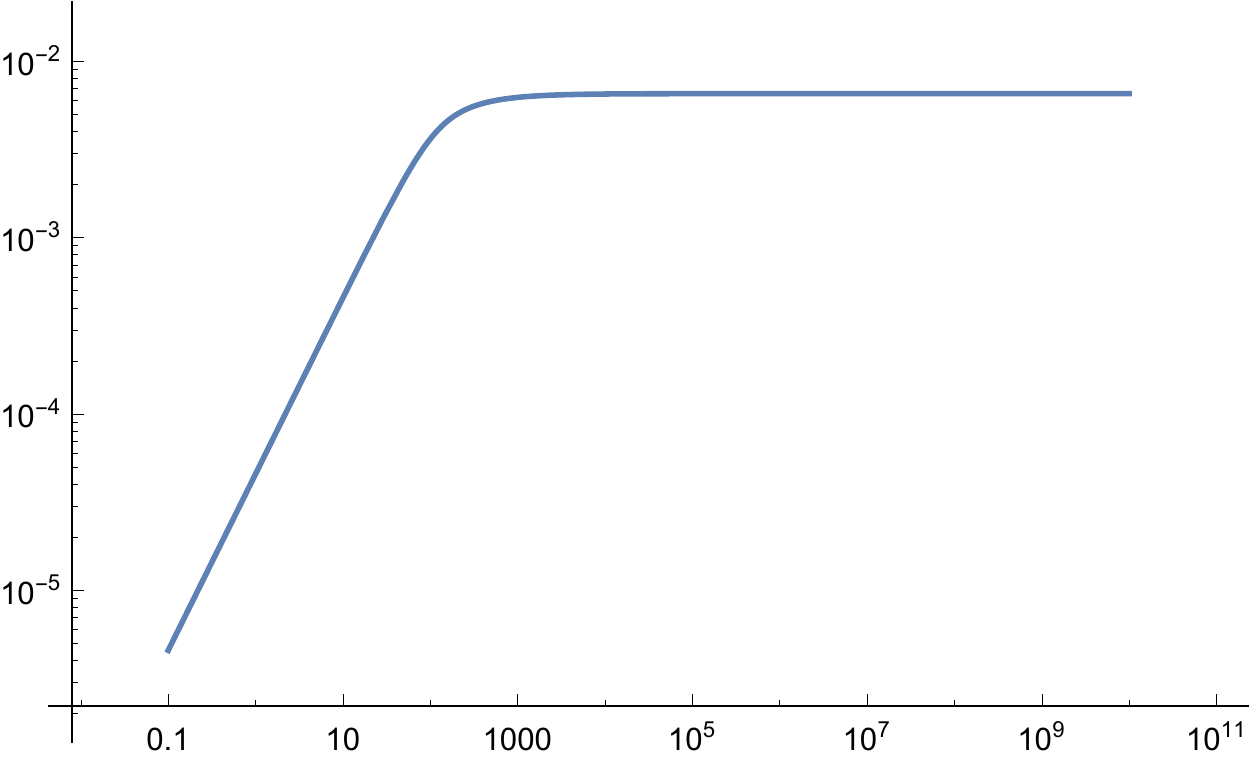}
\caption{The plot of the map $\alpha\mapsto\GG^\infty_\alpha$ in logarithmic scale for $\sigma=0.2$ and $\ell=\frac{\pi}{150}$.}
\label{Ela2}
\end{center}
\end{figure}
\vskip-4mm
A physical interesting case is when $f$ is in resonance with the structure, namely when it is a multiple of an eigenfunction of $\Delta^2$ under the boundary conditions in \eq{eq:Poisson}. Let us briefly recall some known facts from \cite{noi,fergaz}.
The eigenvalues of $\Delta^2$ under the boundary conditions in \eq{eq:Poisson} may be ordered in
an increasing sequence of strictly positive numbers diverging to $+\infty$. Furthermore, the corresponding eigenfunctions form a complete system in $H^2_*$. More precisely, they are identified by two indices $m,j\in\N_+$ and they have one of the following forms:
$$w_{m,j}(x,y)=v_{m,j}(y)\sin(mx) \,\quad \text{with corresponding eigenvalues } \nu_{m,j}\, ,$$
$$\overline w_{m,j}(x,y)=\overline v_{m,j}(y)\sin(mx)\,\quad \text{with corresponding eigenvalues } \mu_{m,j}\, .$$
The $v_{m,j}$ are odd while the $\overline v_{m,j}$ are even and, since $|y|<\ell$ small, one qualitatively has $v_{m,j}(y)\approx\alpha_{m,j}\, y$ and $\overline v_{m,j}(y)\approx\beta_{m,j}$ for
some constants $\alpha_{m,j}$ and $\beta_{m,j}$. This is why the $w_{m,j}$ are called torsional eigenfunctions while the $\overline w_{m,j}$ are called
longitudinal eigenfunctions; see \cite{noi,befega,fergaz}. Here we consider the normalized eigenfunctions
\begin{equation}\label{fmjforce}
f_{m,j}(x,y):=\frac{w_{m,j}(x,y)}{\|w_{m,j}\|_{L^1}}\,.
\end{equation}
We do not consider forcing terms proportional to the longitudinal eigenfunctions since they are even with respect to $y$ and the corresponding
gap functions vanish identically. When $f$ is as in \eq{fmjforce}, we can determine explicitly the gap function.

\begin{theorem}\label{gap_eigen}
Let $\ell>0$, $0<\sigma<1$ be such that the unique positive solution $s>0$ of
\neweq{iff}
\tanh(\sqrt{2}s\ell)=\left(\tfrac{\sigma}{2-\sigma}\right)^2\, \sqrt{2}s\ell
\endeq
is not an integer; let $m,j\in\N$. Let $u_{m,j}$ be the unique solution of \eqref{eq:Poisson} with $f=f_{m,j}$ as in \eqref{fmjforce}, let
$\GG_{m,j}$ and $\GG^{\infty}_{m,j}$ be as in \eqref{gap} and in \eqref{gapinfty} with $u=u_{m,j}$. There exist constants $C_{m,j}=C_{m,j}(\ell)>0$ such that
$\GG_{m,j}(x)=C_{m,j}\sin(mx)$, $\GG^{\infty}_{m,j}=C_{m,j}$, and
\neweq{asympt}
0<\liminf_{\ell\to0}\ \frac{C_{m,j}(\ell)}{\ell^3}\le\limsup_{\ell\to0}\ \frac{C_{m,j}(\ell)}{\ell^3}<\infty\qquad\forall m,j.
\endeq
\end{theorem}

The reason of \eq{iff} will become clear in Section \ref{pttt} where we we give the explicit dependence $C_{m,j}=C_{m,j}(\ell)$, see \eq{Cmj}. Condition \eq{iff} has probability 0 to occur among all possible random choices of $\ell$
and $\sigma$: if it is violated, Theorem \ref{gap_eigen} still holds with a slight modification of the proof. Even if the convergence $\GG^{\infty}_{m,j}(\ell)\to0$ as $\ell\to0$ appears reasonable, a surprise arises from \eq{asympt} which states that
the convergence is at the {\em third} power: this is due to the behavior of the torsional eigenvalues. Take again $\sigma=0.2$ and $\ell=\frac{\pi}{150}$
(two reasonable values for plates modeling the deck of a bridge), by using \eq{Cmj}, we numerically obtained the values in Table \ref{Gmj} below.

\begin{table}[ht]
\begin{center}
\begin{tabular}{|c|c|c|c|c|c|c|c|c|c|c|c|c|}
\hline
$j\backslash m $&$1$&$2$&$3$&$4$&$5$\\
\hline
$1$ &$4.3629\times10^{-3}$&$1.0904\cdot10^{-3}$&$4.8439\cdot10^{-4}$&$2.7229\cdot10^{-4}$&$1.7411\cdot10^{-4}$\\
\hline
$2$ &$4.3566\cdot10^{-8}$&$4.3555\cdot10^{-8}$&$4.3536\cdot10^{-8}$&$4.3509\cdot10^{-8}$&$4.3474\cdot10^{-8}$\\
\hline
$3$ &$4.1216\cdot10^{-9}$&$4.1214\cdot10^{-9}$&$4.1209\cdot10^{-9}$&$4.1203\cdot10^{-9}$&$4.1195\cdot10^{-9}$\\
\hline
$4$ &$9.4439\cdot10^{-10}$&$9.4436\cdot10^{-10}$&$9.4432\cdot10^{-10}$&$9.4426\cdot10^{-10}$&$9.4418\cdot10^{-10}$\\
\hline
$5$ &$3.2251\cdot10^{-10}$&$3.2250\cdot10^{-10}$&$3.2249\cdot10^{-10}$&$3.2248\cdot10^{-10}$&$3.2247\cdot10^{-10}$\\
\hline
\end{tabular}
\par\smallskip
\caption{Numerical values of $\GG^{\infty}_{m,j}(\ell)$ when $\ell=\pi/150$ and $\sigma=0.2$.}\label{Gmj}
\end{center}
\end{table}\vskip-5mm
\noindent
Due to the fact that $w_{1,1}$ is the only torsional eigenfunction with two nodal regions, one expects from a reliable
measure $\GG^\infty$ of the torsional performances of $\Omega$ to satisfy
$$
\max_{m,j\in \N_+}\,\GG^\infty_{m,j}=\GG^\infty_{1,1}\, .
$$
Table \ref{Gmj} supports this conjecture and shows that the maps $m\mapsto\GG^{\infty}_{m,j}$ and $j\mapsto\GG^{\infty}_{m,j}$ are decreasing.
Note that $\GG^{\infty}_{1,1}\approx0.0043<E(\pi/150)\approx0.0065$ which is the value in \eq{Ela} when $\alpha\to+\infty$. Nevertheless, the maximal gap of $f_{1,1}$ is the worst among linear combinations of $f_{m,j}$, namely when a finite number of modes is excited. 


\begin{theorem}\label{particolare}
Let $M\subset\N^+$ be a finite subset of indices and let $m_o=\min\{m\in M\}$.
Let $\overline{u}$ be the unique solution of \eqref{eq:Poisson} with $f(x,y)=\sum_{m\in M}\alpha_m f_{m,j(m)}(x,y)$, where the $f_{m,j}$ are
defined in \eqref{fmjforce} and the $\alpha_m\in\R$ satisfy $\alpha_m\neq0$ and $\sum_{m\in M}|\alpha_{m}|\le1$. Hence, $\|f\|_{L^1}\leq 1$.
Furthermore, assume that the $j=j(m)\in \N$ are such that the map $m\mapsto C_{m,j(m)}$ is nonincreasing with the $ C_{m,j}$ as in Theorem
\ref{gap_eigen}, see also \eqref{Cmj} below. Finally, let $\GG^{\infty}$ be as in \eqref{gapinfty} with $u=\overline{u}$; then
$$\max_{\alpha_m}\GG^\infty=C_{m_o,j(m_o)}$$
where the maximum is taken among all the $\alpha_m$ satisfying the above restrictions.
\end{theorem}

Concerning the existence of a set of indices $\{j(m)\}$ such that the map $m\mapsto C_{m,j(m)}$ is nondecreasing, as required by Theorem \ref{particolare},
Table \ref{Gmj} suggests that this assumption should hold by taking $j(m)= j_0$ for some positive integer $j_0$ fixed.

\section{Proof of Theorem \ref{exponential}}\label{proof Th3}

By linearity, we may take $K_\alpha=1$. Let $\phi\in H^4(0,\pi)$ be the unique solution of
{\small \neweq{equphi}
\left\{\begin{array}{ll}
\phi''''(x)+2\alpha^2\phi''(x)+\alpha^4\phi(x)=g(x)\qquad(0<x<\pi)\\
\phi(0)=\phi(\pi)=\phi''(0)=\phi''(\pi)=0\ .
\end{array}\right.
\endeq}
For $\alpha\not\in\mathbb{N}$, by \eq{serieg} we get $\phi(x)=\serie\frac{\gamma_m}{(m^2-\alpha^2)^2}\smx$, $\phi''\in H^2(0,\pi)$ and the corresponding series converges in $H^2(0,\pi)$ and uniformly. By \eq{equphi}, $\Delta^2[e^{\alpha y}\phi(x)]=e^{\alpha y}g(x)$. Therefore, $v(x,y):=u(x,y)-e^{\alpha y}\phi(x)$, with $u$ solving \eq{eq:Poisson}, satisfies
{\small \neweq{pbv}
\left\{\begin{array}{ll}
\Delta^2v=0 & \mbox{ in }\Omega\\
v=v_{xx}=0 & \mbox{ on }\{0,\pi\}\times(-\ell,\ell)\\
v_{yy}+\sigma v_{xx}=-e^{\pm\alpha\ell}[\alpha^2\phi(x)+\sigma\phi''(x)] & \mbox{ on }(0,\pi)\times\{\pm\ell\}\\
v_{yyy}+(2-\sigma)v_{xxy}=-\alpha e^{\pm\alpha\ell}[\alpha^2\phi(x)+(2-\sigma)\phi''(x)] & \mbox{ on }(0,\pi)\times\{\pm\ell\}\ .
\end{array}\right.
\endeq}

We seek functions $Y_m(y)$ such that $v(x,y)=\serie Y_m(y)\smx$
solves \eq{pbv}. Then $Y_m''''(y)-2m^2Y_m''(y)+m^4Y_m(y)=0$ for $y\in(-\ell,\ell)$, hence $Y_m(y)=A\cosh(my)+B\sinh(my)+Cy\cosh(my)+Dy\sinh(my)$.
By imposing the boundary conditions in \eq{pbv}, we get a $4\times4$ system
that may be split into the two $2\times2$ systems:
{\small \neweq{ADBCsystem}
\left\{\begin{array}{ll}
\!\!\!\!(1\!-\!\sigma)m^2\cosh(m\ell) A+m[2\cosh(m\ell)+(1\!-\!\sigma)m\ell\sinh(m\ell)] D\!=\!
\gamma_m \frac{\sigma m^2\!-\!\alpha^2}{(m^2\!-\!\alpha^2)^2} \cosh(\alpha\ell)\\
\!\!\!\!(\sigma\!-\!1)m^3\sinh(m\ell) A+m^2[(1\!+\!\sigma)\sinh(m\ell)+(\sigma\!-\!1)m\ell\cosh(m\ell)] D\!=\!
\alpha \gamma_m \frac{(2\!-\!\sigma)m^2\!-\!\alpha^2}{(m^2\!-\!\alpha^2)^2} \sinh(\alpha\ell) ,
\end{array}\right.
\endeq
$$
\left\{\begin{array}{ll}
\!\!\!\!(1\!-\!\sigma)m^2\sinh(m\ell) B+m[2\sinh(m\ell)+(1\!-\!\sigma)m\ell\cosh(m\ell)] C\!=\!
\gamma_m \frac{\sigma m^2\!-\!\alpha^2}{(m^2\!-\!\alpha^2)^2} \sinh(\alpha\ell)\\
\!\!\!\!(\sigma\!-\!1)m^3\cosh(m\ell) B+m^2[(1\!+\!\sigma)\cosh(m\ell)+(\sigma\!-\!1)m\ell\sinh(m\ell)] C\!=\!
\alpha \gamma_m \frac{(2\!-\!\sigma)m^2\!-\!\alpha^2}{(m^2\!-\!\alpha^2)^2} \cosh(\alpha\ell) .
\end{array}\right.
$$
}\!
The proof of Theorem \ref{exponential} is so complete.\endproof

\section{Proof of Theorem \ref{gap_limit}}\label{senza rinforzo}
When $f$ is as in \eq{falpha} we have $\gamma_m=0$ for all $m\ge2$ and the systems \eqref{ADBCsystem} yield
{\small $$A(m,\ell,\alpha)=B(m,\ell,\alpha)=C(m,\ell,\alpha)=D(m,\ell,\alpha)=0\qquad\forall m\ge2$$}
while, for $m=1$, $\gamma_1=1$ and by solving \eqref{ADBCsystem} we get
{\small
\begin{eqnarray*}
A&=&\tfrac{(1\!+\!\sigma)(\sigma\!-\!\alpha^2)\sinh(\ell)\cosh(\alpha\ell)+(1\!-\!\sigma)(\alpha^2\!-\!\sigma)\ell\cosh(\ell)\cosh(\alpha\ell)}
{(1\!-\!\sigma)(\alpha^2\!-\!1)^2[(3\!+\!\sigma)\cosh(\ell)\sinh(\ell)-(1\!-\!\sigma)\ell]}\\
\ &\ &+\tfrac{2\alpha(\alpha^2\!+\!\sigma\!-\!2)\cosh(\ell)\sinh(\alpha\ell)+(1\!-\!\sigma)\alpha(\alpha^2\!+\!\sigma\!-\!2)\ell\sinh(\ell)\sinh(\alpha\ell)}
{(1\!-\!\sigma)(\alpha^2\!-\!1)^2[(3\!+\!\sigma)\cosh(\ell)\sinh(\ell)-(1\!-\!\sigma)\ell]}
\end{eqnarray*}\vskip-4mm
\begin{eqnarray}
B&=&\tfrac{(1\!+\!\sigma)(\sigma\!-\!\alpha^2)\cosh(\ell)\sinh(\alpha\ell)+(1\!-\!\sigma)(\alpha^2\!-\!\sigma)\ell\sinh(\ell)\sinh(\alpha\ell)}
{(1\!-\!\sigma)(\alpha^2\!-\!1)^2[(3\!+\!\sigma)\cosh(\ell)\sinh(\ell)+(1\!-\!\sigma)\ell]} \notag\\
\ &\ &+\tfrac{2\alpha(\alpha^2\!+\!\sigma\!-\!2)\sinh(\ell)\cosh(\alpha\ell)+(1\!-\!\sigma)\alpha(\alpha^2\!+\!\sigma\!-\!2)\ell\cosh(\ell)\cosh(\alpha\ell)}
{(1\!-\!\sigma)(\alpha^2\!-\!1)^2[(3\!+\!\sigma)\cosh(\ell)\sinh(\ell)+(1\!-\!\sigma)\ell]} \label{horrible}
\end{eqnarray}
$$C\!=\!\tfrac{\alpha(2\!-\!\sigma\!-\!\alpha^2)\sinh(\ell)\cosh(\alpha\ell)+(\sigma\!-\!\alpha^2)\cosh(\ell)\sinh(\alpha\ell)}
{(\alpha^2\!-\!1)^2[(3\!+\!\sigma)\cosh(\ell)\sinh(\ell)+(1\!-\!\sigma)\ell]},\,D\!=
\!\tfrac{\alpha(2\!-\!\sigma\!-\!\alpha^2)\cosh(\ell)\sinh(\alpha\ell)+(\sigma\!-\!\alpha^2)\sinh(\ell)\cosh(\alpha\ell)}
{(\alpha^2\!-\!1)^2[(3\!+\!\sigma)\cosh(\ell)\sinh(\ell)-(1\!-\!\sigma)\ell]}\, ,
$$}\!
and the explicit form of $u_\alpha$ follows from Theorem \ref{exponential}. In particular, the corresponding gap function $\GG_{\alpha}$ is as in \eq{Gla}.
Hence, $\GG^{\infty}_{\alpha}=E(\ell,\alpha)$, with $E(\ell,\alpha)$ is as in \eq{Elalpha}.\par
In order to compute the limit of $\GG^{\infty}_{\alpha}$ as $\alpha\to+\infty$, we prove

\begin{lemma}\label{limitfalpha}
As $\alpha\to+\infty$ we have that $f_\alpha(x,y)\to\frac{\sin x}{2}\cdot\delta_\ell(y)$ in $H^{-2}_*$
that is,
{\small \neweq{distrib}
\lim_{\alpha\to\infty}\int_\Omega f_\alpha(x,y)v(x,y)\, dxdy=\frac12\int_0^\pi\sin x\, v(x,\ell)\, dx\quad\forall v\in H^2_*\, .
\endeq}\!
\end{lemma}
\begin{proof} Take $v \in H^2_*$, integrating by parts, we have
{\small \begin{eqnarray*}
\int_\Omega f_\alpha(x,y)v(x,y)\, dxdy
= \frac{1}{4\, \sinh(\alpha\ell)}\ipi\sin x\left(\big[e^{\alpha y}v(x,y)\big]_{-\ell}^\ell
-\il e^{\alpha y}v_y(x,y)\, dy\right)\, dx\, .
\end{eqnarray*}}\!
By Lebesgue Theorem, $\lim_{\alpha\to\infty}\, \frac{1}{4\, \sinh(\alpha\ell)}\il e^{\alpha y}v_y(x,y)\, dy=0$. Since, for all $x\in(0,\pi)$, we have that $\lim_{\alpha\to\infty}\, \frac{\big[e^{\alpha y}v(x,y)\big]_{-\ell}^\ell}{4\, \sinh(\alpha\ell)}=\frac{v(x,\ell)}{2}$, hence \eq{distrib} follows.\quad$\Box$
\end{proof}

By letting $\alpha\to+\infty$, the constants in \eq{horrible} have the following asymptotic behavior:
{\small $$A\sim\frac{2\cosh(\ell)+(1-\sigma)\ell\sinh(\ell)}{(1\!-\!\sigma)[(3\!+\!\sigma)\cosh(\ell)\sinh(\ell)-(1\!-\!\sigma)\ell]}\,
\frac{e^{\alpha\ell}}{2\alpha}=:\overline A\,\frac{e^{\alpha\ell}}{\alpha}\, ,$$
$$B\sim\frac{2\sinh(\ell)+(1-\sigma)\ell\cosh(\ell)}{(1\!-\!\sigma)[(3\!+\!\sigma)\cosh(\ell)\sinh(\ell)+(1\!-\!\sigma)\ell]}\,
\frac{e^{\alpha\ell}}{2\alpha}=:\overline B\,\frac{e^{\alpha\ell}}{\alpha}\, ,$$
$$C\sim-\frac{\sinh(\ell)}{(3\!+\!\sigma)\cosh(\ell)\sinh(\ell)+(1\!-\!\sigma)\ell}\, \frac{e^{\alpha\ell}}{2\alpha}=:-\overline C\,\frac{e^{\alpha\ell}}{\alpha}\, ,$$
$$D\sim-\frac{\cosh(\ell)}{(3\!+\!\sigma)\cosh(\ell)\sinh(\ell)-(1\!-\!\sigma)\ell}\, \frac{e^{\alpha\ell}}{2\alpha}=:-\overline D\,\frac{e^{\alpha\ell}}{\alpha}\, .$$}
Set $\overline u\left(x,y\right):=\left[\overline A\cosh(y)+\overline B\sinh(y)-\overline Cy\cosh(y)-\overline Dy\sinh(y)\right]\frac{\sin x}{2}$. Clearly, one has $u_\alpha(x,y)\to\overline u\left(x,y\right)$ a.e. in $\Omega$ as $\alpha\to+\infty$. Moreover, we show

\begin{lemma}\label{limitfalpha2}
As $\alpha\to+\infty$ we have that $$u_\alpha(x,y)\to\overline u\left(x,y\right)\quad\mbox{ in }H_*^2\,.$$
\end{lemma}
\begin{proof}
By definition, $\|f_\alpha\|_{L^1}=1$, and by Lemma \ref{limitfalpha} $f_\alpha \to \overline f:=\frac{\sin x}{2}\cdot\delta_\ell(y)$ in $H^{-2}_*$. By \eq{lax-milg} with $f=f_\alpha $ and with $f=\overline f$, there exists a (unique) $\hat u\in H^2_*$ such that $(\hat u,v)_{H^2_*}=\langle \overline f,v\rangle$ for all $v\in H^2_*$ and $\|u_\alpha-\hat u\|_{H^2_*}\leq \|f_\alpha-\overline f\|_{H^{-2}_*}\,.$
Hence, $u_\alpha(x,y)\!\to\!\hat u\left(x,y\right)$ in $H^2_*$ and, in turn, $\hat u=\overline u$.\quad$\Box$\end{proof}

The gap function corresponding to $\overline u$ is $\overline \GG (x)= \overline u\left(x,\ell\right)-\overline u\left(x,-\ell\right)=E(\ell)\,\sin x\,,$ where $E(\ell)$ is as in \eq{El}. By Lemma \ref{limitfalpha2}, as $\alpha\to+\infty$, we have
$$\max_{x\in [0,\pi]}|\GG_\alpha\left(x\right)\!-\! \overline \GG (x)|\leq \max_{x\in [0,\pi]}|u_\alpha\left(x,\ell\right)\!-\!\overline u\left(x,\ell\right)|+
\max_{x\in [0,\pi]}|u_\alpha\left(x,-\ell\right)\!-\!\overline u\left(x,-\ell\right)| \to 0,$$
and, by \eq{Gla}, $\GG^{\infty}_{\alpha} \to \overline \GG^{\infty}=E(\ell)$. 
Finally, since
{\small $$B=\frac{e^{\alpha\ell}}{\alpha}\left[\overline B+ \frac{(1-\sigma)\ell \sinh (\ell)-(1+\sigma)\cosh(\ell)}{2\alpha(1\!-\!\sigma)[(3\!+\!\sigma)\cosh(\ell)\sinh(\ell)+(1\!-\!\sigma)\ell]}+o\left(\frac{1}{\alpha} \right)\right] \text{ as } \alpha\to+\infty\,,$$
$$C=-\frac{e^{\alpha\ell}}{\alpha}\left[\overline C+ \frac{\cosh(\ell)}{2\alpha[(3\!+\!\sigma)\cosh(\ell)\sinh(\ell)+(1\!-\!\sigma)\ell]}+o\left(\frac{1}{\alpha} \right)\right] \text{ as } \alpha\to+\infty\,,$$}\!
we get the asymptotic in \eq{Ela}. This concludes the proof of Theorem \ref{gap_limit}.

\section{Proofs of Theorems \ref{gap_eigen} and \ref{particolare}}\label{pttt}

We state some properties of the eigenvalues and eigenfunctions of $\Delta^2$ by
slightly improving Theorem 2.1 and Proposition 2.2 in \cite{noi}. By \eq{iff}, two cases may occur. If
{\small \neweq{iff1}
\tanh(\sqrt{2}m\ell)>\left(\tfrac{\sigma}{2-\sigma}\right)^2\, \sqrt{2}m\ell
\endeq}\!
then $m$ is small and the torsional eigenfunction $w_{m,j}$ with $j\ge1$ is given by
{\small \neweq{explicit}
w_{m,j}(x,y)\!=\!\left[\big[\nu_{m,j}^{1/2}-\!(1\!-\!\sigma)m^2\big] \tfrac{\sinh\Big(y\sqrt{\nu_{m,j}^{1/2}+m^2}\Big)}{\sinh\Big(\ell\sqrt{\nu_{m,j}^{1/2}+m^2}\Big)}
+\big[\nu_{m,j}^{1/2}+\!(1\!-\!\sigma)m^2\big]\tfrac{\sin\Big(y\sqrt{\nu_{m,j}^{1/2}-m^2}\Big)}{\sin\Big(\ell\sqrt{\nu_{m,j}^{1/2}-m^2}\Big)}\right]\sin(mx)
\endeq}\!
where the corresponding eigenvalue $\nu_{m,j}$ is the $j$-th solution $\lambda_j>m^2$ of the equation
{\small \neweq{zeroeq}
\sqrt{\sqrt{\lambda}\!-\!m^2}\big[\sqrt{\lambda}\!+\!(1\!-\!\sigma)m^2\big]^2\tanh(\ell\sqrt{\sqrt{\lambda}\!+\!m^2})\!=
\!\sqrt{\sqrt{\lambda}\!+\!m^2}\big[\sqrt{\lambda}\!-\!(1\!-\!\sigma)m^2\big]^2\tan(\ell\sqrt{\sqrt{\lambda}\!-\!m^2});
\endeq}\!
For any $m\geq 1$ and $j\ge1$ we have $\nu_{m,j}>m^4$ and
$\ell\sqrt{\nu_{m,j}^{1/2}-m^2}/\pi\not\in\mathbb{N}$, so that the functions in \eq{explicit} are well-defined. Related to \eq{zeroeq}, we consider the function
{\small \begin{eqnarray*}
H(s) &:=& \sqrt{s^2\!-\!m^4}\big[s\!-\!(1\!-\!\sigma)m^2\big]^2
\left\{\left(\frac{s\!+\!(1\!-\!\sigma)m^2}{s\!-\!(1\!-\!\sigma)m^2}\right)^2\, \frac{\tanh(\ell\sqrt{s\!+\!m^2})}{\sqrt{s\!+\!m^2}}\!-
\frac{\tan(\ell\sqrt{s\!-\!m^2})}{\sqrt{s\!-\!m^2}}\right\}\\
\ &=:& \sqrt{s^2\!-\!m^4}\big[s\!-\!(1\!-\!\sigma)m^2\big]^2\, Z(s)
\end{eqnarray*}}\!
with $s\neq m^2+\frac{\pi^2}{\ell^2}\left(\frac12 +k\right)^2$, $k\in\N$.
In each of the subintervals of definition for $Z$ (and $s>m^2$), the maps
$s\mapsto\frac{s\!+\!(1\!-\!\sigma)m^2}{s\!-\!(1\!-\!\sigma)m^2}$, $s\mapsto\frac{\tanh(\ell\sqrt{s\!+\!m^2})}{\sqrt{s\!+\!m^2}}$, and
$s\mapsto-\frac{\tan(\ell\sqrt{s\!-\!m^2})}{\sqrt{s\!-\!m^2}}$
are strictly decreasing, the first two being positive.  Since, by \eq{iff1},
$\lim_{s\to m^2}Z(s)=\left(\frac{2-\sigma}\sigma \right)^2\, \frac{\tanh(\sqrt2 \, \ell m)}{\sqrt2 \, m}-\ell>0$, $Z$ starts positive, ends up negative and it is strictly decreasing in any subinterval, it admits exactly one zero there, when $\tan(\ell\sqrt{s\!-\!m^2})$ is positive. Hence, $H$ has exactly
one zero on any interval and we have proved
\neweq{intervals}
\nu_{m,j}=\lambda_j\in\left(\big(m^2+\tfrac{\pi^2}{\ell^2}(j-1)^2\big)^2,\big(m^2+\tfrac{\pi^2}{\ell^2}(j-\tfrac12 )^2\big)^2\right)
\quad\forall j\ge1.
\endeq

Slightly different is the second case. If
{\small \neweq{iff2}
\tanh(\sqrt{2}m\ell)<\left(\tfrac{\sigma}{2-\sigma}\right)^2\, \sqrt{2}m\ell\,,
\endeq}\!
$m$ is large, $(1-\sigma^2)m^4<\nu_{m,1}<m^4$, and the first torsional eigenfunction is
{\small \neweq{explicit2}
w_{m,1}(x,y)\!=\!\left[\big[\nu_{m,1}^{1/2}-\!(1\!-\!\sigma)m^2\big] \tfrac{\sinh\Big(y\sqrt{\nu_{m,1}^{1/2}+m^2}\Big)}{\sinh\Big(\ell\sqrt{\nu_{m,1}^{1/2}+m^2}\Big)}
+\big[\nu_{m,1}^{1/2}+\!(1\!-\!\sigma)m^2\big]\tfrac{\sinh\Big(y\sqrt{m^2-\nu_{m,1}^{1/2}}\Big)}{\sinh\Big(\ell\sqrt{m^2-\nu_{m,1}^{1/2}}\Big)}\right]\sin(mx)\,.
\endeq}\!
For $j\ge2$, $\nu_{m,j}>m^4$ and the eigenfunctions are still given by \eq{explicit}. Now $Z$ has no zero in $(m^4,\big(m^2+\tfrac{\pi^2}{4\ell^2}\big)^2)$.
However, for $j\ge2$, \eq{intervals} still holds.
We may now define
{\small \begin{equation}\label{Cmj}
C_{m,j}(\ell):=\frac{4}{\nu_{m,j}^{1/2}\, \|w_{m,j}\|_{L^1}}\,.
\end{equation}}\!
The solution of \eqref{eq:Poisson} with $f=f_{m,j}$ is $u_{m,j}(x,y)=\frac{w_{m,j}(x,y)}{\nu_{m,j} \|w_{m,j}\|_{L^1} }$. Hence, from \eq{explicit} and \eq{Cmj},
$u_{m,j}(x,\pm\ell)=\pm \frac{C_{m,j}}{2}\, \sin(mx)$ and the gap function is $\GG_{m,j}(x)=C_{m,j}\sin(mx)$ with $\GG^{\infty}_{m,j}=C_{m,j}$. We now prove \eq{asympt} with $j\ge2$. From \eq{explicit},
{\small \begin{eqnarray*}
\frac{\|w_{m,j}\|_{L^1}}{4} \ge
\int_0^\gamma\left[\big[\nu_{m,j}^{1/2}+(1\!-\! \sigma)m^2\big]\, \tfrac{\sin\Big(y\sqrt{\nu_{m,j}^{1/2}-m^2}\Big)}{\Big|\sin\Big(\ell\sqrt{\nu_{m,j}^{1/2}-m^2}\Big)\Big|}-                               \big[\nu_{m,j}^{1/2}-(1\!-\!\sigma)m^2\big]\, \tfrac{\sinh\Big(y\sqrt{\nu_{m,j}^{1/2}+m^2}\Big)}{\sinh\Big(\ell\sqrt{\nu_{m,j}^{1/2}+m^2}\Big)}\right]\!dy
\end{eqnarray*}}\!
where $\gamma=\pi/\sqrt{\nu_{m,j}^{1/2}-m^2}<\ell$ in view of \eq{intervals}. By computing the integrals we get
{\small \begin{eqnarray*}
\frac{\|w_{m,j}\|_{L^1}}{4} \ge
2\tfrac{\nu_{m,j}^{1/2}+(1-\sigma)m^2}{\sqrt{\nu_{m,j}^{1/2}-m^2}}-\tfrac{\nu_{m,j}^{1/2}+(1-\sigma)m^2}{\sqrt{\nu_{m,j}^{1/2}-m^2}}
=\tfrac{\nu_{m,j}^{1/2}+(1-\sigma)m^2}{\sqrt{\nu_{m,j}^{1/2}-m^2}}\, .
\end{eqnarray*}}\!
By using \eq{intervals} several more times, we then infer that
\neweq{dalbasso}
\exists c>0\quad\mbox{s.t.}\quad\nu_{m,j}^{1/2}\|w_{m,j}\|_{L^1}\ge c\, \ell^{-3}\, .
\endeq
Let us prove the converse inequality. From \eq{explicit} we infer that
{\small \begin{eqnarray*}
\frac{\|w_{m,j}\|_{L^1}}{4} \le
\tfrac{\nu_{m,j}^{1/2}-(1-\sigma)m^2}{\sqrt{\nu_{m,j}^{1/2}+m^2}}+
\tfrac{\nu_{m,j}^{1/2}+(1-\sigma)m^2}{\sqrt{\nu_{m,j}^{1/2}-m^2}}\,
\tfrac{2j}{\Big|\sin\Big(\ell\sqrt{\nu_{m,j}^{1/2}-m^2}\Big)\Big|}\, .
\end{eqnarray*}}\!
We claim that there exists $\eta>0:$ $\Big|\sin\Big(\ell\sqrt{\nu_{m,j}^{1/2}-m^2}\Big)\Big|>\eta$ for all $\ell>0$.
If not, up to a subsequence $
\omega(\ell):=\ell\sqrt{\nu_{m,j}^{1/2}-m^2}-\pi(j-1)\to0$ as $\ell\to0$.
By replacing into \eq{zeroeq}, we obtain $\omega(\ell)\approx\tanh(\pi(j-1))$, a contradiction.
With this claim,
$$\frac{\|w_{m,j}\|_{L^1}}{4}\le\tfrac{\nu_{m,j}^{1/2}-(1-\sigma)m^2}{\sqrt{\nu_{m,j}^{1/2}+m^2}}+
\tfrac{2j}{\eta}\, \tfrac{\nu_{m,j}^{1/2}+(1-\sigma)m^2}{\sqrt{\nu_{m,j}^{1/2}-m^2}}\, .$$
By using \eq{intervals} several more times, we then infer that
\neweq{dallalto}
\exists c>0\quad\mbox{s.t.}\quad\nu_{m,j}^{1/2}\|w_{m,j}\|_{L^1}\le c\, \ell^{-3}\, .
\endeq
Finally, by combining \eq{Cmj} with \eq{dalbasso} and \eq{dallalto}, we obtain \eq{asympt} for $j\ge2$. The case $j=1$ is simpler. In both cases \eq{iff1} and \eq{iff2}, $w_{m,1}$ does
not change sign. Therefore, we do not need to restrict to $(0,\gamma)$. The estimates are then similar to the case $j\ge2$
for \eq{dalbasso}. For \eq{dallalto} we proceed as for $j\ge2$: if \eq{iff2} holds, the proof is straightforward
since only hyperbolic functions are involved while if \eq{iff1} holds,
{\small \begin{eqnarray*}
\frac{\|w_{m,1}\|_{L^1}}{4} \le
\tfrac{\nu_{m,1}^{1/2}-(1-\sigma)m^2}{\sqrt{\nu_{m,1}^{1/2}+m^2}}+
\ell\, [\nu_{m,1}^{1/2}+(1-\sigma)m^2]\le c\, \nu_{m,1}^{3/4}
\end{eqnarray*}}\!
and we obtain \eq{dallalto}. This concludes the proof of Theorem \ref{gap_eigen}.\par\medskip\noindent
{\bf Proof of Theorem \ref{particolare}.} The function {\small $\overline{u}(x,y)=\sum_{m\in M}\frac{\alpha_m\, w_{m,j(m)}(x,y)}{\nu_{m,j(m)}\, \|w_{m,j(m)}\|_{L^1}}$}
solves \eqref{eq:Poisson} with $f$
as given in the statement. By the explicit form of the $w_{m,j(m)}$, see \eq{explicit} and \eq{explicit2}, we have
{\small $$\overline{u}(x,\pm\ell)=\pm\sum_{m\in M}\tfrac{\alpha_m\, C_{m,j(m)}}{2}\, \sin(mx)\quad\mbox{and}\quad\GG(x)=\sum_{m\in M}\alpha_m C_{m,j(m)}\, \sin(mx)\,.$$}
In particular, we get
{\small $$\GG^{\infty}\leq\sum_{m\in M}|\alpha_{m}|\, C_{m,j(m)}\,.$$}
Let $N$ be the number of integers in the set $M$. By induction one can prove that if $\{\delta_m\}_{m\in M} $ is a nonincreasing set of positive numbers and
$S_N:=\{\textbf{x}=(x_1,...,x_N)\in \R^N\,:\,x_m>0\, \text{ for }x\in M\text{ and }\sum_{m\in M}x_m\leq 1 \}$, then
{\small $$ \max_{\textbf{x}\in S_N}  \sum_{m\in M}\delta_m x_m= \delta_{m_o}\,.$$}
In our case, $x_m=|\alpha_m|$ and $\delta_m=C_{m,j(m)}$. Then the proof of Theorem \ref{particolare} is completed by noting that the function
$f_{m_o,j(m_o)}$ satisfies the assumption of Theorem \ref{particolare} with $\alpha_{m_o}=1$ and, by Theorem \ref{gap_eigen}, the $L^{\infty}$-norm
of the corresponding gap function is $C_{m_o,j(m_o)}$.
\par\medskip\noindent	
{\small \textbf{Acknowledgments.}
The first and second author are partially supported by the Research Project FIR (Futuro in Ricerca) 2013 \emph{Geometrical and
qualitative aspects of PDE's}. The third author is partially supported by the PRIN project {\em Equazioni alle derivate parziali di tipo ellittico e parabolico:
aspetti geometrici, disuguaglianze collegate, e applicazioni}. The three authors are members of the Gruppo Nazionale per l'Analisi Matematica, la Probabilit\`a
e le loro Applicazioni (GNAMPA) of the Istituto Nazionale di Alta Matematica (INdAM).
}

\par


\end{document}